\documentclass[coll]{imsart}

\usepackage{amsthm,amsmath,amsfonts,natbib}
\usepackage{url}

\usepackage[]{fontenc}
\usepackage[latin1]{inputenc}
\usepackage{verbatim}
\usepackage{prettyref}
\usepackage{float}
\usepackage{amsbsy}
\usepackage{graphicx}
\usepackage{amssymb}

\makeatletter
\numberwithin{equation}{section}
\numberwithin{figure}{section}
\theoremstyle{plain}
\newtheorem{thm}{Theorem}[section]
  \theoremstyle{plain}
  \newtheorem{defn}[thm]{Definition}
  \theoremstyle{plain}
  \newtheorem{prop}[thm]{Proposition}
  \theoremstyle{plain}
  
  \theoremstyle{plain}
  \newtheorem{cor}[thm]{Corollary}
 \theoremstyle{definition}
  \newtheorem{example}[thm]{Example}
  \theoremstyle{definition}
  \newtheorem{rem}[thm]{Remark}
  \newcounter{casectr}
  \newenvironment{caseenv}
  {\begin{list}{{\itshape\ Case} \arabic{casectr}.}{%
   \setlength{\leftmargin}{\labelwidth}
   \addtolength{\leftmargin}{\parskip}
   \setlength{\itemindent}{\listparindent}
   \setlength{\itemsep}{\medskipamount}
   \setlength{\topsep}{\itemsep}}
   \setcounter{casectr}{0}
   \usecounter{casectr}}
  {\end{list}}

\allowdisplaybreaks[1]

\usepackage{verbatim}

\usepackage{prettyref}
\newrefformat{eq}{\textup{(\ref{#1})}}
\newrefformat{lem}{Lemma~\ref{#1}}
\newrefformat{pro}{Proposition~\ref{#1}}
\newrefformat{cor}{Corollary~\ref{#1}}
\newrefformat{thm}{Theorem~\ref{#1}}
\newrefformat{cha}{Chapter~\ref{#1}}
\newrefformat{sec}{Section~\ref{#1}}
\newrefformat{sub}{Subsection~\ref{#1}}
\newrefformat{table}{Table~\ref{#1}}
\newrefformat{tab}{Table~\ref{#1}}
\newrefformat{fig}{Figure~\ref{#1}}
\newrefformat{fn}{Footnote~\ref{#1}}
\newrefformat{rem}{Remark~\ref{#1}}
\newrefformat{exa}{Example~\ref{#1}}
\newrefformat{appcha}{Appendix~\ref{#1}}

\usepackage{hyperref}
\hypersetup{
colorlinks,%
citecolor=black,%
filecolor=black,%
linkcolor=black,%
urlcolor=black
}

\@ifundefined{showcaptionsetup}{}{%
 \PassOptionsToPackage{caption=false}{subfig}}
\usepackage{subfig}
\makeatother

\begin{document}

\global\long\def\degree#1{#1^{\circ}}
\global\long\def\mathfun#1{\,\mathrm{#1}}
\global\long\def\prob{\mathfun{Pr}}
\global\long\def\edist{\sim}
\global\long\def\stoless{\le_{\mathrm{st}}\:}
\global\long\def\stolesstr{<_{\mathrm{st}}\:}
\global\long\def\stogr{\ge_{\mathrm{st}}\:}
\global\long\def\stogrtr{>_{\mathrm{st}}\:}
\global\long\def\uprevexch{\leftrightarrow_{\mathrm{RE}}^{+}\,}
\global\long\def\dnrevexch{\leftrightarrow_{\mathrm{RE}}^{-}\,}
\global\long\def\isrevexch{\rightarrow_{\mathrm{RE}}\,}
\global\long\def\arerevexch{\leftrightarrow_{\mathrm{RE}}\,}
\global\long\def\mycondition#1{\mathrm{C}_{#1}}
\global\long\def\reverseexch#1{\mathrm{RE}_{#1}}
\newcommand{\tube}{\Xi}
\def\eqd{\buildrel d\over=}
\def\be{\mathbf{e}}
\def\bX{\mathbf{X}}

\begin{frontmatter}

\title{Reverse Exchangeability and Extreme Order Statistics}
\runtitle{Reverse Exchangeability}

\author{\fnms{Yindeng} \snm{Jiang}\corref{}\ead[label=e1]{yindeng@uw.edu}}
\address{University of Washington Investment Management
Seattle, Washington 98105 
\printead{e1}}
\affiliation{University of Washington}

\author{\fnms{Michael D.} \snm{Perlman}\corref{}\ead[label=e2]{michael@stat.washington.edu}}
\address{Department of Statistics 
University of Washington 
Seattle, WA 98195-4322 
\printead{e2}}
\affiliation{University of Washington}

\runauthor{Jiang and Perlman}


\begin{abstract}
For a bivariate random vector $(X, Y)$, symmetry conditions are presented that yield stochastic orderings
among $|X|$, $|Y |$, $|\max(X, Y )|$, and $|\min(X, Y )|$. Partial extensions of these
results for multivariate random vectors $(X_1, . . . ,X_n)$ are also given.
\end{abstract}


\end{frontmatter}

\section{Introduction}

\citet{Jiang2009factor} introduced a new estimator of value-at-risk
(among other risk and performance measures) for investment funds with
short performance histories. In deriving its large sample variance,
Jiang made use of the following identity (with some rearrangement;
see \citet[Pg. 106, Eq. (3.6.2.4)]{Jiang2009factor}):
\begin{equation}
\Phi_{2}(x,x;\rho)-\Phi_{2}(-x,-x;\rho)=\Phi(x)-\Phi(-x),\ \ \forall x\ge0,\label{introeq1}
\end{equation}
 where $\Phi(\cdot)$ is the cumulative distribution function (cdf)
of the standard normal distribution and $\Phi_{2}(\cdot,\cdot;\rho)$
is the cdf of the standard bivariate normal distribution with correlation
$\rho$.

The result \eqref{introeq1} was somewhat unexpected because the left
hand side is seemingly dependent on $\rho$. Note that the left hand
side is in fact the cdf of $|\max(X,Y)|$, while the right hand side
is the cdf of $|X|$, where $(X,Y)$ is distributed as the standard
bivariate normal distribution with correlation $\rho$. Hence
\eqref{introeq1} implies that \begin{equation}
|\max(X,Y)|\eqd|X|.\label{introeq2}\end{equation}

It is natural to wonder whether this simple but elegant result extends
to bivariate distributions other than the standard bivariate normal
distribution. \prettyref{thm:Rev-exch-2} shows that it does hold
for a broad range of bivariate distributions that are \textit{reverse
exchangeable}.

Next we consider multivariate distributions. For any sequence $X_{1},X_{2},\dots$
of random variables, clearly $\max(X_{1},\dots,X_{n})$ is nondecreasing in $n$,
but this need not be true for $|\max(X_{1},\dots,X_{n})|$: simply
consider a non-random sequence that begins with $-1,0$. Furthermore,
as illustrated by \prettyref{exa:coordinateaxes}, \eqref{introeq2}
need not hold even for multivariate distributions with strong symmetries.
In Theorems \ref{thm:Rev-exch-n} and \ref{thm:Rev-exch-n3}, however,
it is shown that $|\max(X_{1},\dots,X_{n})|$ is stochastically nondecreasing
in $n$ under either of two fairly non-restrictive multivariate extensions
of reverse exchangeability.

A series of examples are presented that illustrate the general results.

\section{Reverse Exchangeability for Bivariate Distributions}
\begin{defn}
\label{Def1} The bivariate random vector $(X,Y)$ is called \emph{reverse
exchangeable (RE)} if $(X,Y)\eqd(-Y,-X)$, that is, $(X,Y)$ and $(-Y,-X)$
are identically distributed.
\end{defn}
Reverse exchangeability simply means that the joint distribution of
$(X,Y)$ is symmetric about the line $y=-x$. Recognizing this allows
us to state the condition in terms of simple reflection. Imagine rotating
the plane clockwise by $\degree{45}$, so the symmetry line $y=-x$
becomes the vertical axis. The point $(X,Y)$ is rotated to
\begin{equation*}
(U,\,V):=\left(\frac{X+Y}{\sqrt{2}},\frac{X-Y}{\sqrt{2}}\right).
\end{equation*}
Then we have the following result:\footnote{Condition \eqref{Ucondition} can be stated as $(U,V)\eqd(-U,V)$,
but we use conditional distributions since this allows for a
natural generalization -- see Definition \ref{Def2}.}
\begin{prop}
\label{pro:Rev-Exchangable}$(X,Y)$ is RE
if and only if the conditional distributions of $U$ and $-U$ given $V$ are the same, i.e.,
\begin{equation}
(U\mid V=v) \eqd (-U\mid V=v),\quad\mathrm{for\ a.e.}\ v\in(-\infty,\infty).\label{Ucondition}
\end{equation}

\end{prop}
Reverse exchangeability is a rather weak condition. For instance,
if $X,Y$ are iid (independent and identically distributed) and $X\eqd-X$ then $(X,Y)$ is RE, but the converse
is not true. A condition weaker than iid but still sufficient for
RE is that the distribution of $(X,Y)$ be ESCI, that is, exchangeable
(E) and sign-change-invariant (SCI).%
\footnote{The ESCI condition is equivalent to group-invariance under the dihedral
group generated by all permutations and sign-changes of coordinates.
See \citet{eaton1977reflection,eaton1982review,eaton1987lectures}
for discussions of group-invariance.%
} Clearly ESCI is strictly stronger than RE since ESCI also implies
symmetry about the line $y=x$, as well as symmetry about both coordinate
axes. It is too strong for our purposes, however, since it is not
satisfied by the class of standard bivariate elliptical distributions
(i.e. with location parameter $(0,0)$ and identical marginals) with
nonzero correlation. Examples of interest include the standard bivariate
normal and bivariate t distributions.

\setlength{\unitlength}{0.5in}

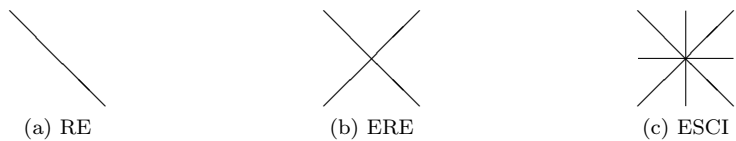
\begin{figure}[H]
\begin{centering}
\subfloat[RE]{\centering{}%
\begin{minipage}[t]{0.33\textwidth}%
\begin{center}
\begin{picture}(1,1)
\put(0,1){\line(1,-1){1}}
\end{picture}
\par\end{center}%
\end{minipage}}\subfloat[ERE]{\centering{}%
\begin{minipage}[t]{0.33\textwidth}%
\begin{center}
\begin{picture}(1,1)
\put(0,1){\line(1,-1){1}} \put(0,0){\line(1,1){1}}
\end{picture}
\par\end{center}%
\end{minipage}}\subfloat[ESCI]{\centering{}%
\begin{minipage}[t]{0.33\textwidth}%
\begin{center}
\begin{picture}(1,1)
\put(0,1){\line(1,-1){1}} \put(0,0){\line(1,1){1}}
\put(0,0.5){\line(1,0){1}} \put(0.5,0){\line(0,1){1}}
\end{picture}
\par\end{center}%
\end{minipage}}
\par\end{centering}

\caption{Three bivariate symmetry conditions.
\label{fig:Reflection}}

\end{figure}

There is a symmetry condition intermediate between RE and ESCI, namely
that $(X,Y)$ is both exchangeable (E) and reverse exchangeable (RE),
designated by ERE, i.e., it is symmetric about the line $y=-x$ and
the line $y=x$. See \prettyref{fig:Reflection} for a comparison
of the three symmetry conditions. All standard bivariate elliptical
distributions are ERE, while any such distribution re-centered at
any point on the line $y=-x$ except the origin satisfies RE but not
ERE.
\begin{example}
\label{exa:ERE-not-ESCI} A class of bivariate distributions that
is ERE but not ESCI arises from sampling without replacement from
a finite set $A$ of real numbers that is symmetric about 0, i.e.,
$A=-A$. If $(X,Y)$ is such a sample from any finite set $A$ with
$|A|\ge2$, then $(X,Y)$ is exchangeable since \begin{equation}
\prob[X=a,Y=b]=\frac{1}{|A|}\cdot\frac{1}{|A|-1},\ \forall\, a,b\in A,\ a\ne b.\label{sampling1}\end{equation}
 If in addition $A=-A$ then $(X,Y)$ is RE:
 \begin{align*}
\prob[-Y=a,-X=b] &\equiv \prob[-X=a,-Y=b]\\
&=\prob[X=a,Y=b]\end{align*}
 by exchangeability and symmetry. Thus $(X,Y)$ is ERE, but it is
not ESCI: for any nonzero $a\in A$, \begin{equation*}
\prob[X=-a,Y=a]=\frac{1}{|A|}\cdot\frac{1}{|A|-1}\ne0=\prob[X=a,Y=a].\label{sampling2}\end{equation*}
 \hfill{}$\square$
\end{example}

Our first result states that if $(X,Y)$ is RE, its absolute marginal
distributions are identical to those of its extreme order statistics:
\begin{thm}
\label{thm:Rev-exch-2}If $(X,Y)$ is reverse exchangeable, then \begin{equation}
|\max(X,Y)|\eqd|\min(X,Y)|\eqd|X|\eqd|Y|\label{Theoremresult}\end{equation}
\end{thm}
\prettyref{thm:Rev-exch-2} follows directly from \prettyref{pro:Rev-exch-2},
which holds under weaker RE conditions.
\begin{defn}
\label{Def2} The bivariate random vector $(X,Y)$ is called \emph{upper
(lower) reverse exchangeable}, designated by \emph{URE (LRE)}, if 
the conditional distributions of $U$ and $-U$ given $V=v>0\ (v<0)$ are the same, i.e.,
\begin{equation*}
(U\mid V=v) \eqd (-U\mid V=v),\quad\mathrm{for\ a.e.}\ v>0\ (v<0).
\end{equation*}
\end{defn}
Clearly RE $\implies$ URE and LRE. The converse need not be true
if \linebreak $\prob[V=0]>0$, i.e. if $\prob[X=Y]>0$, since neither URE nor
LRE ensures that $U\eqd-U\mid V=0$.
%
%
\smallskip

For any $x\ge0$, define the events (see  \prettyref{fig:Plane-partition-1})
\begin{align}
N_x:=&\ \{|X|\le x<Y\},\label{Nsubx}\\
S_x:=&\ \{|X|\le x<-Y\},\label{Ssubx}\\
E_x:=&\ \{|Y|\le x<X\},\label{Esubx}\\
W_x:=&\ \{|Y|\le x<-X\}\label{Wsubx},\\
C_x:=&\ \{|X|\le x,\,|Y|\le x\}\label{Csubx}.
\end{align}
%
%
 For any random variable $Z$, let $F_Z$ denote its cdf. Clearly
\begin{align}
F_{|X|}(x)&=\prob[N_x]+\prob[C_x]+\prob[S_x],\label{FabsX}\\
F_{|Y|}(x)&=\prob[W_x]+\prob[C_x]+\prob[E_x],\label{FabsY}\\
F_{|\max(X,Y)|}(x)&=\prob[W_x]+\prob[C_x]+\prob[S_x],\label{FabsMax}\\
F_{|\min(X,Y)|}(x)&=\prob[N_x]+\prob[C_x]+\prob[E_x].\label{FabsMin}
\end{align}
Therefore,
\begin{align}
F_{|X|}(x)-F_{|\max(X,Y)|}(x)&=\prob[N_x]-\prob[W_x]=F_{|\min(X,Y)|}(x)-F_{|Y|}(x),\label{NminusW}\\
F_{|X|}(x)-F_{|\min(X,Y)|}(x)&=\prob[S_x]-\prob[E_x]=F_{|\max(X,Y)|}(x)-F_{|Y|}(x).\label{SminusE}
\end{align}

\begin{figure}[H]
\begin{centering}
\includegraphics{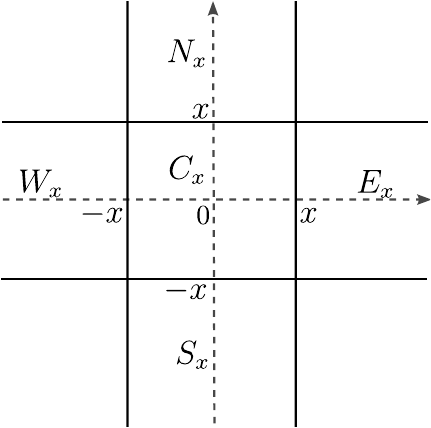}
\par\end{centering}

\caption{The union of the two closed strips $\{|X|\le x\}$ and $\{|Y|\le x\}$.
The regions $N_x$, $S_x$, $E_x$, $W_x$, and $C_x$ are disjoint.
\label{fig:Plane-partition-1}}
\end{figure}
\begin{prop}
\label{pro:Rev-exch-2}
\begin{align*}
 & (i) &  & (X,Y)\ \mathrm{URE} & \implies & |\max(X,Y)|\eqd|X|\ \mbox{and}\ |\min(X,Y)|\eqd|Y|;\\
 & (ii) &  & (X,Y)\ \mathrm{LRE} & \implies & |\max(X,Y)|\eqd|Y|\ \mbox{and}\ |\min(X,Y)|\eqd|X|;\\
 & (iii) &  & (X,Y)\ \mathrm{URE}\mbox{ and }\mathrm{LRE} & \implies &|\max(X,Y)|\eqd|\min(X,Y)|\eqd |X|\eqd|Y|.
 \end{align*}
 \end{prop}
\begin{proof}
Since $(X,Y)$ URE $\Rightarrow$ $\prob[N_x]=\prob[W_x]$ and $(X,Y)$ LRE $\Rightarrow$ $\prob[S_x]=\prob[E_x]$, the results follow from \eqref{NminusW} and \eqref{SminusE}
%
%
\end{proof}

\begin{example}
\label{exa:Normal-iid}If $X,Y$ are iid standard normal random variables,
then $(X,Y)$ is RE. Thus if $M=\min(X,Y)$ or $M=\max(X,Y)$, then
\prettyref{thm:Rev-exch-2} implies $|M|\eqd|X|\eqd|Y|$, hence \begin{equation}
M^{2}\eqd X^{2}\eqd Y^{2}\sim\chi_{1}^{2}.\label{chisquare}\end{equation}
 This result appeared in \citet[Exercise 5.22]{casella2002si}.\hfill{}$\square$
\end{example}
This example can be extended by relaxing normality and/or relaxing
independence:
\begin{example}
\label{exa:Sym-iid-2}If $X,Y$ are iid whose common distribution
is symmetric about 0, then clearly $(X,Y)$ is ESCI, hence RE. For
$M=\max(X,Y)$, \prettyref{thm:Rev-exch-2} implies that $|M|\eqd|X|$.
This can be verified directly from the iid assumption, as follows.

For any $x\ge0$ let $u=\prob[X>x]$. Then \[
\begin{split}\prob[|M|\le x] & =\prob[M\le x]-\prob[M<-x]\\
 & =\left(\prob[X\le x]\right)^{2}-\left(\prob[X<-x]\right)^{2}\\
 & =(1-u)^{2}-u^{2}\\
 & =(1-u)-u\\
 & =\prob[|X|\le x].\end{split}
\]
 Therefore $|M|\eqd|X|$. A similar proof holds if $M=\min(X,Y)$.\hfill{}$\square$
\end{example}

\begin{example}
\label{exa:Bivariate-Normal} (\prettyref{exa:Normal-iid} extended).
Suppose that \begin{equation*}
(X,Y)\sim N_{2}\left(\,(\mu,-\mu),\;\begin{pmatrix}1 & \rho\\
\rho & 1\end{pmatrix}\right),
\end{equation*}
 the bivariate normal distribution with means $\mu$ and $-\mu$ $(-\infty<\mu<\infty)$,
variances 1, and correlation $\rho\in(-1,1)$. Then $(X,Y)$ is not ESCI but
is RE, so \prettyref{thm:Rev-exch-2} implies that for $M=\max(X,Y)$
or $\min(X,Y)$, \begin{equation}
M^{2}\eqd X^{2}\eqd Y^{2}\sim\chi_{1}^{2}(\mu^{2}),\label{noncentralchisquare}\end{equation}
 the noncentral chisquare distribution with noncentrality parameter
$\mu^{2}$, extending \eqref{chisquare}. Note that this result does
not depend on the value of $\rho$.

For $\mu=0$, \eqref{noncentralchisquare} reduces to \eqref{chisquare} as in \prettyref{exa:Normal-iid}, and is equivalent to \eqref{introeq1}.  It seems difficult to verify \eqref{introeq1} directly in this case. \hfill{}$\square$
\end{example}
%
{}
\begin{example}
\label{exa:Bivariate-Elliptical} (\prettyref{exa:Bivariate-Normal}
extended). Suppose that $(X,Y)$ has a bivariate elliptical pdf on
$\mathbb{R}^{2}$ given by
\begin{equation*}
f(x,y)=|\Sigma|^{-1/2}g\left[(x-\mu,\, y+\mu)\,\Sigma^{-1}(x-\mu,\, y+\mu)'\right],
\end{equation*}
where $\Sigma=\sigma^{2}\begin{pmatrix}1 & \rho\\
\rho & 1\end{pmatrix}$.  Then $(X,Y)$ is RE so $M^{2}\eqd X^{2}\eqd Y^{2}$ for all
$\rho\in(-1,1)$. \hfill{}$\square$
\end{example}

\begin{example}
\label{exa:Draw-without-replacement-2}(\prettyref{exa:ERE-not-ESCI}
continued). Suppose that $X,Y$ represent two random draws without
replacement from a finite set $A$ of real numbers that is symmetric about 0, i.e.,
$A=-A$. As
noted in \eqref{sampling1}, $(X,Y)$ is RE, so $|\max(X,Y)|\eqd|X|$
by \prettyref{thm:Rev-exch-2}. If $0\not\in A$ then
\begin{equation*}
\prob[|X|=a]=\frac{2}{|A|},\ \mbox{for}\ a\in A,\ a>0,
\end{equation*}
 while if $0\in A$ then \begin{equation*}
\prob[|X|=a]=\begin{cases}
\frac{1}{|A|}, & a=0,\\
\frac{2}{|A|}, & a\in A,\ a>0,\end{cases}\end{equation*}
 so these are the distributions of $|\max(X,Y)|$ (and of $|\min(X,Y)|$)
as well.\hfill{}$\square$
\end{example}
There is an obvious relation between bivariate RE and bivariate E:
\begin{prop}
$(X,Y)$ is reverse exchangeable if and only if $(X,-Y)$ is exchangeable.
\end{prop}
\noindent Thus \prettyref{thm:Rev-exch-2} has the following corollary:
\begin{cor}
If $(X,Y)$ is exchangeable then \begin{equation*}
|\max(X,-Y)|\eqd|\min(X,-Y)|\eqd|X|\eqd|Y|.\end{equation*}
 \end{cor}
\begin{example}
\label{counterexample} Sign-change invariance is not sufficient for
the conclusion of \prettyref{thm:Rev-exch-2} to hold. Suppose that
$(X,Y)=(1,0)$ and $(-1,0)$, each with probability 1/2. Then $|X|\equiv1$
while $|\max(X,Y)|=0$ and $1$ each with probability 1/2, so \eqref{Theoremresult}
fails, even though $(X,Y)$ is SCI. \hfill{}$\square$
\end{example}

\section{Reverse Exchangeability for Multivariate Distributions}

It is natural to ask if \prettyref{thm:Rev-exch-2} extends to three
or more variables. That is, is \begin{equation*}
|\max(X_{1},\ldots,X_{n-1})|\eqd|\max(X_{1},\ldots,X_{n})|\end{equation*}
 for $n\geq3$ under a general symmetry condition?

The short answer to this question is ``no'', as seen
by the following simple example:
\begin{example}
\label{exa:coordinateaxes} Consider the random vector $\bX_{n}\equiv(X_{1},\dots,X_{n})$
with (discrete) probability distribution specified by \begin{equation*}
\prob[\bX_{n}=\be_{i}]=\prob[\bX_{n}=-\be_{i}]=\frac{1}{2n},\quad i=1,\dots,n,
\end{equation*}
 where $\be_{i}$ denotes the $i$th coordinate unit vector $(0,\dots,0,1_{i},0,\dots,0)$
in $\mathbb{R}^{n}$. Clearly $\bX_{n}$ satisfies the strong symmetry
condition ESCI. However, \begin{equation*}
\prob[|X_{1}|=j]=\begin{cases}
1-\frac{1}{n}, & j=0,\\
\frac{1}{n}, & j=1,\end{cases}\end{equation*}
 while for $l=2,\dots,n$, \begin{equation*}
\prob[|\max(X_{1},\dots,X_{l})|=j]=\begin{cases}
1-\frac{l}{2n}, & j=0,\\
\frac{l}{2n}, & j=1,\end{cases}\end{equation*}
 so that \begin{equation}
|X_{1}|\eqd|\max(X_{1},X_{2})|\stolesstr|\max(X_{1},X_{2},X_{3})|\stolesstr\cdots\stolesstr|\max(X_{1},\ldots,X_{n})|.\label{stochstrict}\end{equation}
 Here $U\stolesstr V$ indicates that $U$ is strictly stochastically
less than $V$, that is, $F_{U}(x)\ge F_{V}(x)$ for all $x$ with
strict inequality for at least one $x$. \hfill{}$\square$
\end{example}
This example shows that \prettyref{thm:Rev-exch-2} does not extend
to three or more dimensions. However, we shall show in Theorems \ref{thm:Rev-exch-n},
\ref{thm:subsupRev-exch-n}, and \ref{thm:Rev-exch-n3} that stochastic
inequalities like those in \eqref{stochstrict} do hold under multivariate
extensions of reverse exchangeability.
\begin{defn}
\label{incrabsmaxmin} The random vector or sequence $(X_{1},\dots,X_{n})$
($n\le\infty$) is said to be \emph{stochastically increasing in absolute
maximum (= SIAMX)} if \begin{equation}
|\max(X_{1},\cdots,X_{l-1})|\stoless|\max(X_{1},\ldots,X_{l})|,\quad\mathrm{for}\ l=2,\dots,n,\label{SIAMX}\end{equation}
 where $U\stoless V$ means that $U$ is stochastically less than
$V$, i.e. $F_{U}(x)\ge F_{V}(x)$ for all $x$. It is \emph{stochastically
increasing in absolute minimum (= SIAMN)} if \eqref{SIAMX} holds
with max replaced by min. It is \emph{strictly SIAMX (= SSIAMX)} or
\emph{strictly SIAMN (= SSIAMN)} if the stochastic inequalities are
strict. It is designated \emph{SIAMX{*} or SSIAMX{*}} if the stochastic
inequalities hold for $l=3,\dots,n$ but for $l=2$ the stochastic
inequality is replaced by $|X_{1}|\eqd|\max(X_{1},X_{2})|$ (e.g.
see \eqref{stochstrict}). It is designated as \emph{SIAMN{*} or SSIAMN{*}}
if, similarly,\linebreak $|X_{1}|\eqd|\min(X_{1},X_{2})|$.
\end{defn}

\begin{defn}
\label{Def3} The random vector $(X_{1},\dots,X_{n})$ is said to
be \emph{RE$(k,l)$} for indices $1\le k<l\le n$ if its distribution is unchanged when $(X_k, X_l)$ is replaced by $(-X_l,-X_k)$, i.e.,
\begin{equation}
(X_{1},\ldots,X_{k},\ldots,X_{l},\dots,X_{n})\eqd(X_{1},\ldots,-X_{l},\ldots,-X_{k},\dots,X_{n}).\label{defhigherRE}
\end{equation}
 Also, $(X_{1},\dots,X_{n})$ is called \emph{RE$(n)$} if it is RE$(k,n)$
for some $k<n$. \end{defn}
\begin{prop}
\label{pro:Rev-exch-n} (i) If $(X_{1},\dots,X_{n})$ is
RE$(k,l)$ then for $m=k$ and for $m=l$, 
\begin{align}
|\max(X_{i}\mid1\le i\le n,\; i\ne m)| & \stoless|\max(X_{i}\mid1\le i\le n)|,\label{REnmax}\\
|\min(X_{i}\mid1\le i\le n,\; i\ne m)| & \stoless|\min(X_{i}\mid1\le i\le n)|.\label{REnmin}
\end{align}
(ii) Strict stochastic inequality holds in \eqref{REnmax}, respectively, in \eqref{REnmin}, if
\begin{align}
\prob[X_{m}>\max(|X_{i}|\mid1\le i\le n,\; i\ne m)]&> 0,\ \mbox{respectively,}\label{strictineqmaxm}\\
\prob[X_{m}<-\max(|X_{i}|\mid1\le i\le n,\; i\ne m)]&> 0.\label{strictineqminm}\end{align}
 \end{prop}
\begin{proof}
(i) Without loss of generality take $(k,l)=(1,n)$ and $m=l=n$, so \eqref{defhigherRE},
\eqref{REnmax}, and \eqref{REnmin} become \begin{align}
(X_{1},X_{2}\dots,X_{n-1},X_{n}) & \eqd(-X_{n},X_{2}\dots,X_{n-1},-X_{1}),\label{defhigherRE1}\\
|\max(X_{1},\ldots,X_{n-1})| & \stoless|\max(X_{1},\ldots,X_{n})|,\label{REnmax1}\\
|\min(X_{1},\ldots,X_{n-1})| & \stoless|\min(X_{1},\ldots,X_{n})|,\label{REnmin1}\end{align}
 respectively. For $x\ge0$ define the event \begin{equation*}
\Omega_{n}(x):=  \left\{ |\max(X_{1},\ldots,X_{n})|\le x\right\}\\ =\left\{ -x\le\max(X_{1},\ldots,X_{n})\le x\right\} .\label{aa}\end{equation*}
 To prove \eqref{REnmax1} we need to show that \begin{equation}
\prob[\Omega_{n}(x)]\le\prob[\Omega_{n-1}(x)].\label{eq:QnQn-1}\end{equation}

For any subset $D\subseteq N:=\{1,\ldots,n\}$, define the event \begin{equation*}
\mathcal{T}_{n}(D)\equiv\mathcal{T}_{n}(D;x):=\{X_{i}<-x\ \forall i\in D\}\cap\{|X_{i}|\le x\ \forall i\notin D\}.\end{equation*}
 Note that the events $\mathcal{T}_{n}(D)$ are disjoint for $D\subseteq N$.
Then \begin{eqnarray}
\Omega_{n}(x) & = & \bigcup_{D\subset N}\mathcal{T}_{n}(D)\nonumber \\
 & = & \left(\bigcup_{D\subset N,\, n\in D}\mathcal{T}_{n}(D)\right)\cup\left(\bigcup_{D\subset N,\, n\notin D}\mathcal{T}_{n}(D)\right)\nonumber \\
 & = & \left(\bigcup_{D\subset N\backslash\{n\}}\mathcal{T}_{n}(D\cup\{n\})\right)\cup\left(\bigcup_{D\subseteq N\backslash\{n\}}\mathcal{T}_{n}(D)\right)\nonumber \\
 & = & \left(\bigcup_{D\subset N\backslash\{n\}}\Big(\mathcal{T}_{n}(D\cup\{n\})\cup\mathcal{T}_{n}(D)\Big)\right)\cup\Big(\mathcal{T}_{n}(N\backslash\{n\})\Big).\label{eq:Qn}\end{eqnarray}
 For any $D\subseteq N\backslash\{n\}$ define \begin{equation*}
\widetilde{\mathcal{T}}_n(D)\equiv\widetilde{\mathcal{T}}_n(D;x):=\{X_{i}<-x\ \forall i\in D\}\cap\{|X_{i}|\le x\ \forall i\notin D,\, i\neq n\}\cap\{X_{n}>x\},\end{equation*}
 also a family of disjoint events. Note too that $\mathcal{T}_{n}(D)\cap\widetilde{\mathcal{T}}_n(D')=\emptyset$
for any $D,D'$. If $D\subset N\backslash\{n\}$, it is straightforward
to verify that \begin{equation*}
\mathcal{T}_{n-1}(D)=\mathcal{T}_{n}(D\cup\{n\})\cup\mathcal{T}_{n}(D)\cup\widetilde{\mathcal{T}}_n(D),\end{equation*}
 a union of three disjoint events. Thus \begin{eqnarray}
& & \Omega_{n-1}(x) \nonumber \\
 & = & \bigcup_{D\subset N\backslash\{n\}}\mathcal{T}_{n-1}(D)\nonumber \\
 & = & \bigcup_{D\subset N\backslash\{n\}}\Big(\mathcal{T}_{n}(D\cup\{n\})\cup\mathcal{T}_{n}(D)\cup\widetilde{\mathcal{T}}_n(D)\Big)\nonumber \\
 & = & \left(\bigcup_{D\subset N\backslash\{n\}}\Big(\mathcal{T}_{n}(D\cup\{n\})\cup\mathcal{T}_{n}(D)\Big)\right)\cup\left(\bigcup_{D\subset N\backslash\{n\}}\widetilde{\mathcal{T}}_n(D)\right),\label{eq:Qn-1}\end{eqnarray}
 where all the events involving $\mathcal{T}_{n}$ and $\widetilde{\mathcal{T}}_n$
are mutually disjoint. But the RE$(1,n)$ condition \eqref{defhigherRE}
implies that \begin{equation}
\prob[\mathcal{T}_{n}(N\backslash\{n\})]=\prob[\widetilde{\mathcal{T}}_n(N\backslash\{1,n\})]\leq\prob\left[\bigcup_{D\subset N\backslash\{n\}}\widetilde{\mathcal{T}}_n(D)\right],\label{strictineqTtilde}\end{equation}
 which, together with \eqref{eq:Qn} and \eqref{eq:Qn-1}, yields
\eqref{eq:QnQn-1} and thence \eqref{REnmax1}. \medskip{}

Now \eqref{REnmin1} follows from \eqref{REnmax1} because \begin{align}
(X_{1},\dots,X_{n})\ \mathrm{is\ RE}(k,l) & \iff(-X_{1},\dots,-X_{n})\ \mathrm{is\ RE}(k,l)\quad\mathrm{and}\label{XiffminusX}\\
|\min(X_{1},\ldots,X_{n})| & \ \ =\ \ \ |\max(-X_{1},\ldots,-X_{n})|.\label{absmaxabsmin}\end{align}

(ii) Because the events $\widetilde{\mathcal{T}}_n(D)$
are disjoint, it follows from \eqref{strictineqTtilde} that strict
inequality holds in \eqref{eq:QnQn-1} iff \begin{equation}
\prob\left[\widetilde{\mathcal{T}}_n(D;x)\right]>0,\ \ \mathrm{for\ some}\ D\subset N\backslash\{n\},\, D\ne N\backslash\{1,n\}.\label{strictineqTtildeD}\end{equation}
 In particular, set $D=\emptyset$ to see that \eqref{strictineqTtildeD}
holds if
\begin{equation}
\prob[|X_{i}|\le x,\ \forall\ i=1,\dots,n-1,\; X_{n}>x]>0.\label{strictineqemptyset}
\end{equation}
Thus a sufficient condition for strict stochastic inequality to hold
in \eqref{REnmax1} is that \eqref{strictineqemptyset} hold for at
least one $x$, which is equivalent\footnote{Since $\{X_{n}>\max(|X_{1}|,\dots,|X_{n-1}|)\} =\cup(\{X_n>x\ge \max(|X_{1}|,\dots,|X_{n-1}|)\}\mid x\in\mathbb{Q}
)$.}
to the condition that
\begin{equation*}
\prob[X_{n}>\max(|X_{1}|,\dots,|X_{n-1}|)]>0,
\end{equation*}
thus confirming \eqref{strictineqmaxm}.
 By \eqref{absmaxabsmin}, it follows that a sufficient condition
for strict stochastic inequality to hold in \eqref{REnmin1} is that
\begin{equation*}
\prob[X_{n}<-\max(|X_{1}|,\dots,|X_{n-1}|)]>0,
\end{equation*}
thereby confirming \eqref{strictineqminm}
 \end{proof}
\begin{rem}
\label{counterexample2} The distribution of $|\max(X_{i}\mid1\le i\le n,\; i\ne m)|$
in \prettyref{pro:Rev-exch-n} is not necessarily the same for $m=k$
and $m=l$. With $n=3$, $k=1$, and $l=2$, consider the random vector
$(X_{1},X_{2},X_{3})$ that assigns probability 1/4 to each of the
four points $(-1,0,0)$, $(0,1,0)$, $(0,-1,1)$, and $(1,0,1)$.
Then this distribution is RE$(1,2)$ but \begin{align*}
\prob[|\max(X_{1},X_{3})|=r]= & \begin{cases}
\frac{1}{2}, & r=0\\
\frac{1}{2}, & r=1,\end{cases}\\
\prob[|\max(X_{2},X_{3})|=r]= & \begin{cases}
\frac{1}{4}, & r=0\\
\frac{3}{4}, & r=1.\end{cases}\end{align*}
 The same is true for $|\min(X_{i}\mid1\le i\le n,\; i\ne m)|$. \hfill{}$\square$
\end{rem}
\prettyref{thm:Rev-exch-2} and \prettyref{pro:Rev-exch-n} yield
the following multivariate result:
\begin{thm}
\label{thm:Rev-exch-n}Let the random vector or sequence $(X_{1},\dots,X_{n})$
be such that $(X_{1},\dots,X_{l})$ is RE$(l)$ for each $l=2,\dots,n$.
Then $(X_{1},\dots,X_{n})$ is SIAMX{*} and SIAMN{*}. It is SSIAMX{*}
or SSIAMN{*} if \begin{align}
\prob[X_{l}>\max(|X_{1}|,\dots,|X_{l-1}|)]>0, & \quad l=3,\dots,n,\label{suffcondstrictineq1ell}\\
\mathrm{or}\quad\prob[X_{l}<-\max(|X_{1}|,\dots,|X_{l-1}|)]>0, & \quad l=3,\dots,n,\label{suffcondstrictineq2ell}\end{align}
 respectively.
\end{thm}
It is easy to see that the discrete multivariate distribution in \prettyref{exa:coordinateaxes}
satisfies condition \eqref{suffcondstrictineq1ell}, thereby confirming
the strict stochastic inequalities in \eqref{stochstrict}. (The same
holds true if max is replaced by min in \eqref{stochstrict}.)
\begin{rem}
For $(X_{1},X_{2})$, RE$(2)$ is simply RE, which is weaker than
ESCI as noted before. For $(X_{1},\dots,X_{n})$ with $n\ge3$, the
conjunction of RE$(2)$, \dots, RE$(n)$ in \prettyref{thm:Rev-exch-n}
is weaker than ESCI in general. Consider, for example, an infinite
sequence $X_{1},X_{2},X_{3},\ldots$ of iid but non-symmetric rvs (random variables).
For any $n\ge2$, \linebreak[1] $(-X_{1},X_{2},X_{3},\ldots,X_{n})$ is RE$(n)$
but not ESCI.\hfill{}$\square$ \end{rem}
\begin{example}
{[}\prettyref{exa:Sym-iid-2} continued{]} \label{exa:Sym-iid-n}If
$X_{1},...,X_{n}$ are iid random variables whose common distribution
is symmetric about 0, then $(X_{1},...,X_{n})$ is ESCI hence RE$(n)$
for every $n\ge2$. Here the conclusions of \prettyref{thm:Rev-exch-n}
can be verified directly: \smallskip{}
 To show that $(X_{1},\dots,X_{n})$ is SIAMX{*}, for any $x\ge0$
set $u_{x}=\prob[X_{i}>x]\le\frac{1}{2}$. Then as in \prettyref{exa:Sym-iid-2},
\begin{equation*}
\prob[\,|\max(X_{1},...,X_{n})|\le x]=(1-u_{x})^{n}-u_{x}^{n},\end{equation*}
 which is decreasing in $n$ since \begin{eqnarray}
(1-u_{x})^{n-1}-u_{x}^{n-1} & \ge & (1-u_{x})^{n}-u_{x}^{n}\nonumber \\
 & \Updownarrow\nonumber\\
(1-u_{x})^{n-1}u_{x} & \ge & u_{x}^{n-1}(1-u_{x})\nonumber \\
 & \Updownarrow\nonumber\\
\left(\frac{1-u_{x}}{u_{x}}\right)^{n-1} & \ge & \frac{1-u_{x}}{u_{x}}.\nonumber
\end{eqnarray}
 The last  inequality
 holds since $\frac{1-u_{x}}{u_{x}}\ge1$.

This inequality is strict if $n\ge3$ and $0<u_{x}<\frac{1}{2}$,
i.e., if $\prob[|X_{i}|\le x]>0$. Thus for such $x$, $\prob[\,|\max(X_{1},...,X_{n})|\le x]$
is strictly decreasing in $n$ for $n\ge2$. A necessary and sufficient
condition for this to hold for at least one $x\ge0$, and therefore
for $(X_{1},\dots,X_{n})$ to be SSIAMX{*}, is that the distribution of
$|X_{i}|$ be non-degenerate. Note that this condition is equivalent
to both \eqref{suffcondstrictineq1ell} and \eqref{suffcondstrictineq2ell}
in this example, so under this condition,
$(X_{1},\dots,X_{n})$ is SSIAMN{*} as well. \hfill{}$\square$
\end{example}
For independent random variables, however, the requirements of identical
distributions and symmetry in \prettyref{exa:Sym-iid-n} are not necessary
for RE$(n)$ to hold:
\begin{example}
\label{exa:Alternate-series}Let $(X_{1},X_{2},\ldots)$ be an infinite
sequence of independent random variables such that \begin{equation}
X_{1}\eqd-X_{2}\eqd X_{3}\eqd-X_{4}\eqd X_{5}\eqd\cdots\ .\label{eq:Alternate-series}\end{equation}
 Then for each $n\ge2$, $(X_{1},...,X_{n})$ is RE$(n)$ with $k=n-1$
(or $n-3,\ n-5,\ldots$). Thus \prettyref{thm:Rev-exch-n} implies
that $(X_{1},X_{2},\ldots)$ is SIAMX{*} and SIAMN{*}. If in addition
both \eqref{suffcondstrictineq1ell} and \eqref{suffcondstrictineq2ell}
hold, then by \prettyref{pro:Rev-exch-n}(ii), $(X_{1},X_{2},\ldots)$
is SSIAMX{*} and SSIAMN{*}. (These results can again be verified directly,
as in \prettyref{exa:Sym-iid-n}.)

In fact, the same conclusions holds if \prettyref{eq:Alternate-series}
is weakened to the condition \begin{equation}
X_{i}\eqd\epsilon_{i}X_{1},\ i=2,3,\ldots,\end{equation}
 where \begin{equation*}
\epsilon_{2}=-1,\ \epsilon_{i}=\pm1,\;\mbox{for}\; i\ge3.\label{eq:Sign-sequence}\end{equation*}
 Now $(X_{1},...,X_{n})$ is RE$(n)$ with either $k=1$ or $k=2$.
\hfill{}$\square$
\end{example}
We now present an example where it seems difficult to circumvent \prettyref{thm:Rev-exch-n}.
Such examples arise when $X_{1},X_{2},\dots$ are not independent.
(Also see Examples \ref{exa:Draw-without-replacement-ncontd} and
\ref{exa:intraclass}.)
\begin{example}
\label{Gaussiansequence} Consider a Gaussian sequence $(X_{1},X_{2},\ldots)$
with $\mathrm{E}(X_{i})=\mu_{i}$, $\mathrm{Var}(X_{i})=\sigma^{2}$,
and $\mathrm{Corr}(X_{i},X_{j})=\rho_{i,j}$. Then $(X_{1},...,X_{n})$
satisfies RE$(n)$ if and only if for some $1\le k(n)\le n-1$, \begin{align}
\mu_{n}= & -\mu_{k(n)},\label{eq:REn-mu}\\
\rho_{n,j}= & -\rho_{k(n),j},\ \forall\, j<n,\ j\ne k(n).\label{eq:REn-rho}\end{align}
 If these conditions hold for every $n=2,3,\ldots$, then \prettyref{thm:Rev-exch-n}
implies that the sequence is SIAMX{*} 
and, by \prettyref{pro:Rev-exch-n}(ii), is SSIAMX{*} if the Gaussian
sequence is nonsingular.

Since $k(n)<n$, the functional iterates $k^{(q)}(n)$ strictly decrease
with $q$. Let $q_{n}$ be the smallest $q$ such that $k^{(q)}(n)=1$;
note that $q_{2}=1$. Thus, if \prettyref{eq:REn-mu} holds for all
$n=2,3,\ldots$ then $\mu_{n}=(-1)^{q_{n}}\mu_{1}$, so the sequence
of means $\boldsymbol{\mu}_{\infty}:=(\mu_{1},\mu_{2},\dots)$ takes
the form \begin{equation*}
\boldsymbol{\mu}_{\infty}=\mu\cdot\big(1,-1,(-1)^{q_{3}},(-1)^{q_{4}},\dots\big)\end{equation*}
 for some scalar $\mu$. \smallskip{}

If \prettyref{eq:REn-rho} holds for all $n=2,3,\ldots$, the structure
of the correlation matrix $\mathbf{R}_{\infty}:=(\rho_{i,j}\mid1\le i,j<\infty)$
is more complicated to describe. We present two special cases:
\begin{caseenv}
\item $k(n)=1$ for each $k\ge2$: Here each $q_{n}=1$ so \begin{equation*}
\boldsymbol{\mu}_{\infty}=\mu\cdot(1,-1,-1,-1,\ldots)\end{equation*}
 and $\mathbf{R}_{\infty}$ has the form \begin{equation*}
\mathbf{R}_{\infty}=\left(\begin{array}{ccccc}
1 & \rho_{1} & \rho_{2} & \rho_{3} & \cdots\\
\rho_{1} & 1 & -\rho_{1} & -\rho_{1} & \cdots\\
\rho_{2} & -\rho_{1} & 1 & -\rho_{2} & \cdots\\
\rho_{3} & -\rho_{1} & -\rho_{2} & 1 & \cdots\\
\vdots & \vdots & \vdots & \vdots & \ddots\end{array}\right).\end{equation*}

\item $k(n)=n-1$ for each $k\ge2$: Here $q_{n}=n-1$ so \begin{equation*}
\boldsymbol{\mu}_{\infty}=\mu\cdot(1,-1,1,-1,1,\ldots)\end{equation*}
 and $\mathbf{R}_{\infty}$ has the form \begin{equation*}
\mathbf{R}_{\infty}=\left(\begin{array}{cccccc}
1 & \rho_{1} & -\rho_{1} & \rho_{1} & -\rho_{1} & \cdots\\
\rho_{1} & 1 & \rho_{2} & -\rho_{2} & \rho_{2} & \cdots\\
-\rho_{1} & \rho_{2} & 1 & \rho_{3} & -\rho_{3} & \cdots\\
\rho_{1} & -\rho_{2} & \rho_{3} & 1 & \rho_{4} & \cdots\\
-\rho_{1} & \rho_{2} & -\rho_{3} & \rho_{4} & 1 & \cdots\\
\vdots & \vdots & \vdots & \vdots & \vdots & \ddots\end{array}\right).\end{equation*}

\end{caseenv}
Thus $(X_{1},X_{2},\ldots)$ is SIAMX{*} and SIAMN{*}, and is SSIAMX{*}
and SSIAMN{*} if the Gaussian sequence is nonsingular. 
\hfill{}$\square$
\end{example}
Lastly, we have the following result for an exchangeable random vector:
\begin{cor}
\label{cor:Exchangeable}If $(X_{1},\ldots,X_{n})$ is exchangeable,
then $(-X_{1},X_{2},\ldots,X_{n})$ is SIAMX{*} and SIAMN{*}. \end{cor}
\begin{proof}
Exchangeability implies that $(-X_{1},\ldots,X_{l})$ is RE$(1,l)$
for $l=2,\dots,n$, so this result follows from \prettyref{thm:Rev-exch-n}. \end{proof}
\begin{example}
\label{Example25} Suppose that $X=(X_{1},...,X_{n})$ represent $n$
random draws (without replacement) from a finite set of real numbers.
Since \linebreak $(X_{1},\ldots,X_{n})$ is exchangeable, it follows from \prettyref{cor:Exchangeable}
that \linebreak $(-X_{1},X_{2},\ldots,X_{n})$ is SIAMX{*} and SIAMN{*}. 
\hfill{}$\square$
\end{example}

\section{Reverse Sub(Super)exchangeability}
\begin{example}
\label{exa:Draw-without-replacement-n}(\prettyref{exa:Draw-without-replacement-2}
extended). Suppose that $X_{1},\dots,X_{n}$ represent $n$ random
draws without replacement from the finite symmetric set $A\subset\mathbb{R}$,
where $n\le|A|$. In \prettyref{exa:ERE-not-ESCI} it was shown that
$(X_{1},X_{2})$ is RE $\equiv$ RE$(2)$. However, $(X_{1},\dots,X_{l})$
is not RE$(l)$ for $3\le l\le n\wedge(|A|-1)$: for example, if $l=3$
and $a,b\in A$, $a,b>0$, $a\ne b$, then $(X_{1},X_{2},X_{3})$
is not RE$(1,3)$: \begin{align*}
0&= \prob[X_{1}=-a,\,X_{2}=-a,\,X_{3}=b]\\
&< \prob[X_{1}=-b,\,X_{2}=-a,\,X_{3}=a]\\
&= \prob[-X_{3}=-a,\,X_{2}=-a,\,-X_{1}=b],\end{align*}
 where the strict inequality holds since $-b,-a,a$ are distinct.
Similarly \linebreak $(X_{1},X_{2},X_{3})$ is not RE$(2,3)$, hence $(X_{1},X_{2},X_{3})$
is not RE$(3)$. Thus the condition of \prettyref{thm:Rev-exch-n}
is not satisfied. Nonetheless, $(X_{1},\dots,X_{n})$ is SSIAMX and
SSIAMN in this example. 
\hfill{}$\square$
\end{example}
To establish this fact we introduce the notions of \emph{ reverse
subexchangeability} and \emph{ reverse superexchangeability}, weaker
conditions than reverse exchangeability. For simplicity we shall restrict
attention to random vectors $(X_{1},\dots,X_{n})$ whose distributions
are determined by $f(x_{1},\dots,x_{n})$, which is either a discrete
probability mass function (pmf) or a probability density function
(pdf) w.r.to Lebesgue measure. We begin with the bivariate case.

\begin{defn}
\label{Def7} The bivariate random vector $(X,Y)$ is called \emph{upper
(lower) reverse subexchangeable}, denoted by \emph{UR$_{\mathrm{E}}$ (LR$_{\mathrm{E}}$)},
if \begin{equation}
f(x,y)\ge f(-y,-x),\ \ \mbox{for}\ |x|<y\ \ (|y|<x).\label{RSEbivariate3}
\end{equation}
The rv $(X,Y)$ is called \emph{upper (lower) reverse superexchangeable},
denoted by \emph{UR$^{\mathrm{E}}$ (LR$^{\mathrm{E}}$)}, if
\begin{equation}
f(x,y)\le f(-y,-x),\ \ \mbox{for}\ |x|<y\ \ (|y|<x).\label{RSEbivariate4}
\end{equation}
\end{defn}

\begin{prop}
\label{pro:Rev-exch-5}
\begin{align*}
 & (i) &  &(X,Y)\ \mathrm{UR}_{\mathrm{E}}\   & \implies &|X|\stoless|\max(X,Y)|\ \mathrm{and}\ |\min(X,Y)|\stoless|Y|;\\
 & (ii) &  & (X,Y)\ \mathrm{LR}_{\mathrm{E}}\   & \implies & |Y|\stoless|\max(X,Y)|\ \mathrm{and}\ |\min(X,Y)|\stoless|X|;\\
 & (iii) &  & (X,Y)\ \mathrm{UR}^{\mathrm{E}}\  & \implies & |X|\stogr|\max(X,Y)|\ \mathrm{and}\ |\min(X,Y)|\stogr|Y|;\\
  & (iv) &  &(X,Y)\ \mathrm{LR}^{\mathrm{E}}\   & \implies & |Y|\stogr|\max(X,Y)|\ \mathrm{and}\ |\min(X,Y)|\stogr|X|.
 \end{align*}
 The stochastic inequalities in (i) (resp., (iii)) are strict if and only if
 \begin{equation}
\prob[\,|X|<Y\,]\ >\ (<)\ \prob[\,X<-|Y|\,].\label{eq:YXneXY}
\end{equation}
 Likewise, the stochastic inequalities in (ii) (resp., (iv))
are strict if and only if
\begin{equation}
\prob[\,|Y|<X\,]\ >\ (<)\ \prob[\,Y<-|X|\,].
\end{equation}
 \end{prop}
\begin{proof}
Because $(X,Y)$ UR$_{\mathrm{E}}$ $\Rightarrow$ $\prob[N_x]\ge \prob[W_x]$ and $(X,Y)$ LR$_{\mathrm{E}}$ $\Rightarrow$ $\prob[S_x]\le \prob[E_x]$, (i) and (ii) follow from \eqref{NminusW} and \eqref{SminusE} respectively. Parts (iii) and (iv) follow similarly with the inequalities reversed.

To establish strict stochastic inequality in (i), define $N:\{|X|<Y\}$ and, for any measurable $A\subseteq N$, define
\begin{equation*}
\tilde A:=\{(-y,-x)\mid (x,y)\in A\},
\end{equation*}
the reflection of $A$ across the line $y=-x$. Note that $\tilde N=\{X<-|Y|\}=:W$ and $\tilde N_x=W_x$ for $x\ge0$ (recall \eqref{Nsubx} and \eqref{Wsubx}).

For any measurable subset $A\subseteq N$, define
\begin{equation*}
\sigma(A):=\prob[A]-\prob[\tilde A],
\end{equation*}
so that (recall \eqref{Nsubx} and \eqref{Wsubx},
\begin{align}
\sigma(N)&=\prob[N]-\prob[W]=\prob[\,|X|<Y\,]-\prob[\,X<-|Y|\,],\label{4.12}\\
\sigma(N_x)&=\prob[N_x]-\prob[W_x].\label{4.13}
\end{align}
Clearly $\sigma$ is a countably additive set function. Since
 $(X,Y)$ is $\mathrm{UR}_{\mathrm{E}}$, we have that
 \begin{equation*}
\sigma(A)\ge0,\ \ \forall\ \mathrm{measurable}\ A\subseteq\Omega,
\end{equation*}
so $\sigma$ is a nonnegative measure. Thus, because $N$ is the countable union
\begin{equation*}
N=\bigcup(N_x\mid x\ge0,\ x\ \mathrm{rational}),
\end{equation*}
it follows that
\begin{equation*}
 \sigma(N)>0\iff \sigma(N_x)>0,\ \mathrm{for\ at\ least\ one\ rational}\ x\ge0.
\end{equation*}
The result now follows from \eqref{4.12}, \eqref{4.13}, and \eqref{NminusW}.
\smallskip

Cases (ii), (iii), and (iv) are treated similarly.
 \end{proof}

\begin{example}
\label{exa:Normal-two-means} (\prettyref{exa:Bivariate-Elliptical} extended).
Suppose that $(X,Y)$ has a bivariate elliptical pdf on
$\mathbb{R}^{2}$ given by
\begin{equation*}
f(x,y)=|\Sigma|^{-1/2}g\left[(x-\mu,\, y-\nu)\,\Sigma^{-1}(x-\mu,\, y-\nu)'\right],
\end{equation*}
 where $\Sigma=\sigma^{2}\begin{pmatrix}1 & \rho\\\rho & 1\end{pmatrix}$, $-1<\rho<1$. Assume that
$g$ is \emph{nonincreasing and strictly positive} on $[0,\infty)$. (This includes the case where $(X,Y)\sim N_2((\mu,\nu),\,\Sigma)$). After some algebra we find that
\begin{equation*}
(x+y)(\mu+\nu)\ge0\implies f(x,y)\ge f(-y,-x),
\end{equation*}
regardless of the value of $\rho$, so  
 \begin{equation}
\mu+\nu>0\implies (X,Y)\ \mathrm{is\ UR}_{\mathrm{E}}\ \mathrm{and\ LR}_{\mathrm{E}}.\label{101}
\end{equation}
Furthermore  $(X-\mu,Y-\nu)$ is RE, so
\begin{align*}
\prob[\,X<-|Y|\,]&= \prob[\,(X-\mu)+\mu<-|(Y-\nu)+\nu|\,]\\
&= \prob[\,-(Y-\nu)+\mu<-|-(X-\mu)+\nu|\,]\\
&= \prob[\,Y-(\mu+v)>|X-(\mu+\nu)|\,]\\
&=\prob[\,Y>X,\ X+Y>2(\mu+\nu)\,].
 \end{align*}
Thus, if $\mu+\nu>0$ then
\begin{align*}
 & \prob[\,|X|<Y\,]-\prob[\, X<-|Y|\,]\\
= & \prob[\, Y>X,\ Y>-X\,]-\prob[\, Y>X,\ X+Y>2(\mu+\nu)\,]\\
= & \prob[\, Y>X,\ X+Y>0\,]-\prob[\, Y>X,\ X+Y>2(\mu+\nu)\,]\\
= & \prob[\, Y>X,\ 0<X+Y\le2(\mu+\nu)\,],
\end{align*}
which is strictly positive since $g$ is strictly positive on $[0,\infty)$.
It follows from \prettyref{pro:Rev-exch-5}  that
\begin{equation}
\mu+\nu>0\implies |\min(X,Y)|\stolesstr\begin{matrix}|X|\\|Y|\end{matrix}\stolesstr|\max(X,Y)|. \label{102}
\end{equation}

Similarly,
\begin{align}
\mu+\nu<0\implies & (X,Y)\ \mathrm{is\ UR}^{\mathrm{E}}\ \mathrm{and\ LR}^{\mathrm{E}}\label{103}\\
\implies & |\max(X,Y)|\stolesstr\begin{matrix}|X|\\|Y|\end{matrix}\stolesstr|\min(X,Y)|. \label{104}
\end{align}
(Note that if $\mu+\nu=0$ then $(X,Y)$ is RE so \prettyref{exa:Bivariate-Elliptical}
applies, hence these stochastic inequalities become stochastic equalities.)
\hfill{}$\square$
\end{example}

\begin{example}
\label{exa:Sym-iid-21}Suppose that $(X,Y)$ has joint pmf or pdf
given by \begin{equation*}
f(x,y)=g(x)\, h(y),
\end{equation*}
 where in addition, $g$ and $h$ are symmetric about 0, i.e., $g(x)=g(-x)$
and $h(y)=h(-y)$. Thus $X$ and $Y$ are independent and SCI, but
neither E nor RE if $g\ne h$. Here
\begin{align}
f(x,y)\ge f(-y,-x)&\iff g(|x|) h(|y|)\ge g(|y|)h(|x|),\label{ghsymmetric}\\
 &\iff f_{|X|}(|x|)f_{|Y|}(|y|)\ge  f_{|X|}(|y|)f_{|Y|}(|x|),\label{ghsymmetric2}
\end{align}
where $f_{|X|}$ and $f_{|Y|}$ denote the pmfs or pdfs of $|X|$ and $|Y|$, respectively.

%
 %

 Now specialize to the case where $f_{|X|}=f_{\theta'}$
and $f_{|Y|}=f_{\theta''}$ are pmfs or pdfs in a one-parameter family
$\{f_{\theta}\}$ of pmfs or pdfs on $[0,\infty)$ with
\emph{strictly monotone-increasing likelihood ratio}.
Then for $\theta'<\theta''$,
the inequalities \eqref{ghsymmetric}-\eqref{ghsymmetric2} hold whenever $0<x<y$
while the opposite inequalities hold whenever $0<y<x$.
Thus $(X,Y)$ is UR$_{\mathrm{E}}$ and LR$^{\mathrm{E}}$ if $\theta'<\theta''$.
 Furthermore by symmetry
 \begin{align*}
 \prob[\,|X|<Y]&=\prob[Y<-|X|\,]=\frac{1}{2}\prob[\,|X|<|Y|\,],
 \\
 \prob[\,X<-|Y|\,]&=\prob[\,|Y|<X]=\frac{1}{2}\prob[\,|Y|<|X|\,],
 \end{align*}
 and
 \begin{equation*}
 \prob[\,|X|<|Y|\,]>\prob[\,|Y|<|X|\,]
 \end{equation*}
 by the strict monotone likelihood ratio assumption for $|X|$ and $|Y|$. Thus by \prettyref{pro:Rev-exch-5} 
(also note $(X,Y)\eqd (-X,-Y)$),
\begin{equation}
|X|\stolesstr|\min(X,Y)|\eqd |\min(-X,-Y)| \eqd |\max(X,Y)|\stolesstr|Y|.\label{strictMLR}
\end{equation}

 The scale-parameter families $\{N(0,\theta)\mid \theta>0\}$ and $\{C(0,\theta)\mid \theta>0\}$
of centered normal and Cauchy pdfs satisfy the assumptions of this
example, hence satisfy \eqref{strictMLR} when $0<\theta'<\theta''$. \hfill{}$\square$
\end{example}

\begin{example}
\label{exa:Bivariate-Normal21} (\prettyref{exa:Bivariate-Elliptical}
extended). Suppose that $(X,Y)$ has a centered bivariate elliptical
pdf on $\mathbb{R}^{2}$ given by \begin{align*}
f(x,y)&= |\Sigma|^{-1/2}g\left[(x,\, y)\,\Sigma^{-1}(x,\, y)'\right]
\\
&= |\Sigma|^{-1/2}g\left[d(\sigma,\tau,\rho)(x^{2}\tau^{2}-\rho\sigma\tau xy+y^{2}\sigma^{2})\right],
\end{align*}
 where $\Sigma=\begin{pmatrix}\sigma^{2} & \rho\sigma\tau\\
\rho\sigma\tau & \tau^{2}\end{pmatrix}$ is positive definite and $d(\cdot,\cdot,\cdot)>0$. Assume that
$g$ is \emph{nonincreasing} on $[0,\infty)$. (This includes the case $(X,Y)\sim N_2((0,0),\,\Sigma)$).  Then
\begin{equation*}
(y^{2}-x^{2})(\tau^{2}-\sigma^{2})\ge 0\implies f(x,y)\ge f(-y,-x),
\end{equation*}
regardless of the value of $\rho$, so
 \begin{equation}
 \tau^{2}>\sigma^{2}\implies (X,Y)\ \mathrm{is\ UR}_{\mathrm{E}}\ \mathrm{and\ LR}^{\mathrm{E}}.\label{tausigineq1}
\end{equation}
Furthermore, $(\tau X,\sigma Y)$ is RE, so
\begin{align*}
\prob[\,X<-|Y|\,]
&=  \prob[\,-\sigma Y/\tau<-|\tau X/\sigma|\,]\\
&=  \prob[\,Y>(\tau/\sigma)^2 |X|\,].
 \end{align*}
Thus if $\tau^2>\sigma^2$ then
\begin{equation*}
 \prob[\,|X|<Y\,]-\prob[\,X<-|Y|\,]=\prob[\,(\tau/\sigma)^2 |X|\ge Y>|X|\,],
 \end{equation*}
which is strictly positive for any bivariate elliptical distribution.
It follows from \prettyref{pro:Rev-exch-5}  that (also note $(X,Y)\eqd (-X,-Y)$)
\begin{equation}
\tau^2>\sigma^2\implies |X|\stolesstr|\min(X,Y)|
\eqd |\min(-X,-Y)| \eqd|\max(X,Y)|\stolesstr|Y|,\label{tausigineq2}
\end{equation}
regardless of the value of $\rho$. Similarly,
\begin{equation}
\tau^2<\sigma^2\implies |Y|\stolesstr|\min(X,Y)|\eqd|\max(X,Y)|\stolesstr|X|.\label{tausigineq3}
\end{equation}
\hfill{}$\square$
\end{example}

We now turn to the multivariate case; take $n\ge3$ for the remainder
of this section.
\begin{defn}
\label{Def8} The random vector $(X_{1},\dots,X_{n})$ is said to
be \emph{UR$_{\mathrm{E}}(k,l)$} for indices $1\le k<l\le n$ if 
$f(x_1,\dots,x_n)$ decreases when $(x_k, x_l)$ is replaced by $(-x_l, -x_k)$, i.e.,
\begin{equation}
f(x_{1},\dots,x_{k},\dots,x_{l},\dots,x_{n})\ge f(x_{1},\dots,-x_{l},\dots,-x_{k},\dots,x_{n}),\label{defhigherULRE}
\end{equation}
 whenever $|x_{k}|<x_{l}$ and $x_{i}<-|x_{k}|$ for all $i\ne k,l$.
It is \emph{LR$_{\mathrm{E}}(k,l)$} if \eqref{defhigherULRE} holds
whenever $|x_{l}|<x_{k}$ and $x_{i}<-|x_{l}|$ for all $i\ne k,l$.
The rv $(X_{1},\dots,X_{n})$ is called \emph{UR$_{\mathrm{E}}(l)$
(LR$_{\mathrm{E}}(l)$)} if it is UR$_{\mathrm{E}}(k,l)$ (LR$_{\mathrm{E}}(k,l)$)
for some $k\ne l$. Then \emph{UR$^{\mathrm{E}}(k,l)$, LR$^{\mathrm{E}}(k,l)$,
UR$^{\mathrm{E}}(n)$, and LR$^{\mathrm{E}}(n)$)} are defined like
their counterparts but with the inequality reversed in \eqref{defhigherULRE}.
\smallskip{}
\end{defn}

\begin{prop}
\label{pro:Rev-subsupexch-n}For $1\le k<l\le n$,
\begin{align*}
& (i) & &(X_{1},\dots,X_{n})\ \mathrm{UR}_{\mathrm{E}}(k,l)\  \implies|\max(X_{i}\mid i\ne l)|\stoless|\max(X_{1},\dots,X_{n})|;\\
& (ii) & & (X_{1},\dots,X_{n})\ \mathrm{LR}_{\mathrm{E}}(k,l)\  \implies|\max(X_{i}\mid i\ne k)|\stoless|\max(X_{1},\dots,X_{n})|;\\
& (iii) & & (X_{1},\dots,X_{n})\ \mathrm{UR}^{\mathrm{E}}(k,l)\  \implies|\min(X_{i}\mid i\ne k)|\stoless|\min(X_{1},\dots,X_{n})|;\\
& (iv) & & (X_{1},\dots,X_{n})\ \mathrm{LR}^{\mathrm{E}}(k,l)\  \implies|\min(X_{i}\mid i\ne l)|\stoless|\min(X_{1},\dots,X_{n})|.
\end{align*}
The stochastic inequalities in (i) -- (iv) become strict under the four conditions \begin{align}
\prob[X_{l}>\max(|X_{i}|\mid1\le i\le n,\; i\ne l)]> & \ 0,\label{strictineqmaxl}\\
\prob[X_{k}>\max(|X_{i}|\mid1\le i\le n,\; i\ne k)]> & \ 0,\label{strictineqmaxk}\\
\prob[X_{k}<-\max(|X_{i}|\mid1\le i\le n,\; i\ne k)]> & \ 0,\label{eq:strictineqminm}\\
\prob[X_{l}<-\max(|X_{i}|\mid1\le i\le n,\; i\ne l)]> & \ 0,\label{eq:strictineqminm-1}\end{align}
 respectively.
 \end{prop}
\begin{proof}
The proof of (i) is identical to the proof of \eqref{REnmax}
in \prettyref{pro:Rev-exch-n}, except that in \eqref{strictineqTtilde}
the equality (=) is replaced by inequality ($\le$), which is justified
by \eqref{defhigherULRE}. The implications (ii)-(iv)  are established
in similar fashion. An argument similar to that used for \prettyref{pro:Rev-exch-n}(ii) verifies the conditions for strict stochastic inequality.
\end{proof}

%
%

\prettyref{pro:Rev-subsupexch-n} yields the following theorem:
\begin{thm}
\label{thm:subsupRev-exch-n} If $(X_{1},\dots,X_{l})$ is UR$_{\mathrm{E}}(l)$
 for $l=2,\dots,n$, then\linebreak $(X_{1},\dots,X_{n})$
is SIAMX. If in addition
\begin{align}
&\prob[X_2>|X_1|\,]>\prob[X_1<-|X_2|\,]\label{suffcondstrictineq1ell0}\\
\mathrm{and}\quad&\prob[X_l>\max(|X_1|,\dots,|X_{l-1}|)]>0,\ \ \mathrm{for}\ l=3,\dots,n,\label{suffcondstrictineq1ell3}
\end{align}
then $(X_{1},\dots,X_{n})$ is SSIAMX.
%
%
If $(X_{1},\dots,X_{l})$ is LR$^{\mathrm{E}}(l)$  for $l=2,\dots,n$, then $(X_{1},\dots,X_{n})$
is SIAMN. If in addition
\begin{align}
&\prob[\,|X_2|<X_1]<\prob[X_2<-|X_1|\,]\label{suffcondstrictineq1ell00}\\
\mathrm{and}\quad&\prob[X_l<-\min(|X_1|,\dots,|X_{l-1}|)]>0,\ \ \mathrm{for}\ l=3\dots,n,\label{suffcondstrictineq2ell4}
\end{align}
 then $(X_{1},\dots,X_{n})$ is SSIAMN.
 %
%
 \end{thm}

\begin{example}
\label{exa:Draw-without-replacement-ncontd}(\prettyref{exa:Draw-without-replacement-n}
continued). Let $X_{1},\dots,X_{n}$ be $n$ random draws without
replacement from the finite symmetric set $A\subset\mathbb{R}$, with
$n\le|A|$. It was seen in \prettyref{exa:Draw-without-replacement-n}
that $(X_{1},\dots,X_{l})$ need not be RE$(l)$ for $3\le l\le n$. However, it is UR$_{\mathrm{E}}(l)$, which is seen as follows:

The pmf $f(x_{1},\dots,x_{l})$ takes the value $c:=1/(|A|(|A|-1)\cdots(|A|-l+1))$
on the range \begin{equation}
\tilde{A}^{l}:=\{(x_{1},\dots,x_{l})\mid x_{1},\dots,x_{l}\in A\ \mathrm{and\ are\ mutually\ distinct}\},\label{Antilde}\end{equation}
 and is 0 for $(x_{1},x_{2},\dots,x_{l})\notin\tilde{A}^{l}$. Thus,
to verify via \eqref{defhigherULRE} that $(X_{1},\dots,X_{l})$ is
RE$(1,l)$ it suffices to show that $(x_{1},x_{2},\dots,x_{l})\in\tilde{A}^{l}$
whenever \linebreak $(-x_{l},x_{2},\dots,x_{l-1},-x_{1})\in\tilde{A}^{l}$ and
$|x_{1}|<x_{l}$ and $x_{i}<-|x_{1}|$ for all $i\ne k,l$. These
two strict inequalities imply that $x_{1},x_{i},x_{l}$ are mutually
distinct for all $i=2,\dots,l-1$, while $x_{2},\dots,x_{l-1}$ are
distinct by \eqref{Antilde}. The assertion follows since $-x_{i}\in A\implies x_{i}\in A$
by the symmetry of $A$.

It follows from \prettyref{thm:subsupRev-exch-n} that $(X_{1},\dots,X_{n})$
is SIAMX and, by symmetry, is SIAMN. Furthermore, conditions \eqref{suffcondstrictineq1ell3}
and \eqref{suffcondstrictineq2ell4} hold unless $l=n=|A|$, so $(X_{1},\dots,X_{n})$
is SSIAMX and SSIAMN, except possibly for the case $l=n$ when $n=|A|$. \hfill{}$\square$
\end{example}

\begin{example}
\label{exa:intraclass} (\prettyref{exa:Bivariate-Elliptical} extended).
Let $(X_{1},\dots,X_{n})$ have a centered multivariate elliptical
distribution with pdf of the form
\begin{equation*}
f(x_{1},\dots,x_{n})=|\Sigma|^{-1/2}g\left[(x_{1},\dots,x_{n})\,\Sigma^{-1}(x_{1},\dots,x_{n})'\right],
\end{equation*}
 where $\Sigma\equiv\{\sigma_{ij}\}$ is positive definite with \textit{intraclass
correlation structure}:\linebreak $\sigma_{ii}=\sigma^{2}$ and $\sigma_{ij}=\sigma^{2}\rho$
for all $1\le i\ne j\le n$, where $-1/(n-1)<\rho<1$. Then $f$ has
the form \begin{equation*}
f(x_{1},\dots,x_{n})=|\Sigma|^{-1/2}g\left[c(\sigma,\rho)\cdot\!\!\sum_{1\le i\le n}x_{i}^{2}-\rho\, d(\sigma,\rho)\cdot\!\!\!\!\sum_{1\le i<j\le n}\!\! x_{i}x_{j}\right],
\end{equation*}
 where $c(\cdot,\cdot)>0$ and $d(\cdot,\cdot)>0$.  Assume that $g$ is \emph{nonincreasing}
on $[0,\infty)$. Then for $\rho\ne0$,
\begin{align*}
 & \ \rho\,(x_{1}+x_{n})(x_{2}+\cdots+x_{n-1})\ge0\\
\implies & f(x_{1},\dots,x_{n})\ge f(-x_{n},x_{2},\dots,x_{n-1},-x_{1}),
\end{align*}
hence
 \begin{equation}\label{cases}
(X_{1},\dots,X_{n})\ \mathrm{is}\ \begin{cases}
\mathrm{UR}_{\mathrm{E}}(1,n),\  & \mathrm{if}\ \rho<0,\\
\mathrm{LR}^{\mathrm{E}}(1,n),\  & \mathrm{if}\ \rho>0.\end{cases}
\end{equation}
 Since $(X_{1},X_{2})$ is RE for all $\rho$, it
follows from Theorems \ref{thm:Rev-exch-2} and \ref{thm:subsupRev-exch-n}
that $(X_{1},\dots,X_{n})$ is SSIAMX{*} when $\rho<0$ and is SSIAMN{*}
when $\rho>0$.
However, $(X_{1},\dots,X_{n})\eqd(-X_{1},\dots,-X_{n})$ for all values of $\rho$,
so\linebreak $|\max(X_{1},\ldots,X_{n})|\eqd|\min(X_{1},\ldots,X_{n})|$. Thus
$(X_{1},\dots,X_{n})$ is SSIAMX{*} and SSIAMN{*} for all $\rho$.
\smallskip{}

Note that $(X_{1},\dots,X_{n})$ is also exchangeable, so
by \prettyref{cor:Exchangeable} and \prettyref{thm:subsupRev-exch-n}, $(-X_{1},X_{2},\ldots,X_{n})$ is
SIAMX{*} and SIAMN{*}. \hfill{}$\square$
\end{example}

\section{Results for Independent Symmetric Random Variables.}

As noted earlier, if $X_{1},\dots,X_{n}$ are iid and symmetric about
0 then $(X_{1},\dots,X_{n})$ is ESCI, hence RE$(k,l)$ for all $1\le k<l\le n$,
so the conclusions of \prettyref{thm:Rev-exch-n} hold. Under independence,
however, stochastic comparisons for the extreme order statistics can
be obtained under a weaker assumption than identical distributions,
namely stochastic ordering. This is illustrated by the following bivariate
result:
\begin{thm}
\label{bivariateindepsymm} Suppose that $X$ and $Y$ are independent
and symmetric about 0. If $|X|\stoless|Y|$ then \begin{equation}
|X|\stoless|\min(X,Y)|\eqd|\max(X,Y)|\stoless|Y|.\label{bivariatestoless}\end{equation}
 These stochastic inequalities are strict iff $|X|\stolesstr|Y|$. \end{thm}
\begin{proof}
We may either pattern the proof after that of \prettyref{pro:Rev-exch-2} or use
this direct approach: \smallskip{}

If $F$ and $G$ are the cdfs of $X$ and $Y$, respectively, then
the cdfs of $|X|$ and $|Y|$ are $2F-1$ and $2G-1$ on $[0,\infty)$.
Furthernore, the cdf of $|\max(X,Y)|$ is \begin{align*}
FG-(1-F)(1-G)= & F+G-1
\\
= & \frac{1}{2}[(2F-1)+(2G-1)],\end{align*}
 the average of $2F-1$ and $2G-1$. But $|X|\stoless|Y|$ implies
that $2F-1\ge2G-1$, hence this average lies in the interval $[2G-1,\,2F-1]$,
and strictly inside this interval for some $x$ iff $|X|\stolesstr|Y|$.
This yields the stochastic inequalities in \eqref{bivariatestoless}
and the statement regarding strict stochastic inequality. The equality
in \eqref{bivariatestoless} follows by symmetry.
\end{proof}

Note that the stochastic ordering assumption $|X|\stoless|Y|$ in
Theorem \ref{bivariateindepsymm} is weaker than the monotone likelihood
ratio assumption in \prettyref{exa:Sym-iid-21}.
For example, let $|X|$ take the values $0$ and $1$ with probability
$1/2$ each, and let $|Y|$ take the values $0$, $1$, and $2$ with
probabilities $1/2$, $1/4$, and $1/4$ respectively. \medskip{}

In three or more dimensions we have the following partial complement
to \prettyref{pro:Rev-subsupexch-n} and \prettyref{thm:subsupRev-exch-n}:
\begin{thm}
\label{thm:Rev-exch-n3}Let $X_{1},\dots,X_{n}$ be independent symmetric
random variables. If $|X_{k}|\stoless|X_{l}|$ for a pair $(k,l)$
with $1\le k<l\le n$, then \begin{align}
|\max(X_{i}\mid1\le i\le n,\; i\ne l)|\stoless & |\max(X_{i}\mid1\le i\le n)|,\label{REnmax7}\\
|\min(X_{i}\mid1\le i\le n,\; i\ne l)|\stoless & |\min(X_{i}\mid1\le i\le n)|.\label{REnmin7}\end{align}
 If $|X_{k}|\stolesstr|X_{l}|$, the stochastic inequalities in \eqref{REnmax7}
and \eqref{REnmin7} are strict. \smallskip{}

Therefore, if $|X_{1}|\stoless\cdots\stoless|X_{n}|$ then $(X_{1},\dots,X_{n})$
is SIAMX and SIAMN. If $|X_{1}|\stolesstr\cdots\stolesstr|X_{n}|$,
then $(X_{1},\dots,X_{n})$ is SSIAMX and SSIAMN. \end{thm}
\begin{proof}
Without loss of generality take $k=1$ and $l=n$. Set $F_{i}(x)=\prob[X_{i}\le x]$,
the cdf of $X_{i}$, and set $\bar{F}_{i}=1-F_{i}$. Similarly, let
$G_{i}$ denote the cdf of $|X_{i}|$ and $\bar{G}_{i}=1-G_{i}$.
By symmetry, for $i=1,\dots,n$ and $x\ge0$ we have \begin{equation}
\bar{F}_{i}(x)=\frac{1}{2}\bar{G}_{i}(x).\label{FbarGbar}\end{equation}
 If we let $H_{n}$ denote the cdf of $|\max(X_{i}\mid1\le i\le n)|$,
then by independence,
\begin{equation}
H_{n}(x)=\prod_{i=1}^{n}F_{i}(x)-\prod_{i=1}^{n}\bar{F}_{i}(x),\quad\mathrm{for}\ x\ge0.\label{cdfHFFbar}\end{equation}
 Therefore, by \eqref{cdfHFFbar} and \eqref{FbarGbar},
 \begin{align*}
 & \ H_{n}(x)-H_{n-1}(x)\\
= & \ \bar{F}_{1}(x)F_{n}(x)\prod_{i=2}^{n-1}\bar{F}_{i}(x)-F_{1}(x)\bar{F}_{n}(x)\prod_{i=2}^{n-1}F_{i}(x)\\
= & \ \frac{1}{4}\left[\bar{G}_{1}(x)(2-\bar{G}_{n}(x))\prod_{i=2}^{n-1}\bar{F}_{i}(x)-(2-\bar{G}_{1}(x))\bar{G}_{n}(x)\prod_{i=2}^{n-1}F_{i}(x)\right].\end{align*}
 Since $\bar{F}_{i}(x)\le F_{i}(x)$ for $i=2,\dots,n-1$ and $|X_{1}|\stoless|X_{n}|\implies\bar{G}_{1}(x)\le\bar{G}_{n}(x)$,
it follows that $H_{n}(x)\le H_{n-1}(x)$ for $x\ge0$, which confirms
\eqref{REnmax7}. \smallskip{}

If $|X_{1}|\stolesstr|X_{n}|$ then $\bar{G}_{1}(x)<\bar{G}_{n}(x)$
for some $x\ge0$. Since $F_{i}(x)\ge\frac{1}{2}>0$ for $i=2,\dots,n-1$,
it follows that $H_{n}(x)<H_{n-1}(x)$ for this $x$, hence the stochastic
inequality in \eqref{REnmax7} is strict. The remaining assertions
are straightforward. \end{proof}
\begin{example}
Let $X_{1},\dots,X_{n}$ be independent symmetric random variables,
and let $|X_{i}|$ ($i=1,\ldots,n$) take the values $0,1,\ldots,i$
with probability $\frac{1}{2},\frac{1}{2i},\ldots,\frac{1}{2i}$,
respectively. Then it is clear that $|X_{1}|\stolesstr\cdots\stolesstr|X_{n}|$.
It follows from \prettyref{thm:Rev-exch-n3} that $(X_{1},\dots,X_{n})$
is SSIAMX and SSIAMN. \hfill{}$\square$
\end{example}

\bibliographystyle{imsart-nameyear}
\bibliography{backfill}

\begin{thebibliography}{5}

\bibitem[\protect\citeauthoryear{Casella and Berger}{2002}]{casella2002si}
\begin{bbook}[author]
\bauthor{\bsnm{Casella},~\bfnm{G.}\binits{G.}} \AND
  \bauthor{\bsnm{Berger},~\bfnm{R.~L.}\binits{R.~L.}}
(\byear{2002}).
\btitle{{Statistical Inference}}.
\bpublisher{Duxbury Pacific Grove, CA}.
\end{bbook}
\endbibitem

\bibitem[\protect\citeauthoryear{Eaton}{1982}]{eaton1982review}
\begin{barticle}[author]
\bauthor{\bsnm{Eaton},~\bfnm{M.~L.}\binits{M.~L.}}
(\byear{1982}).
\btitle{{A review of selected topics in multivariate probability
  inequalities}}.
\bjournal{The Annals of Statistics}
\bvolume{10}
\bpages{11--43}.
\end{barticle}
\endbibitem

\bibitem[\protect\citeauthoryear{Eaton}{1987}]{eaton1987lectures}
\begin{bbook}[author]
\bauthor{\bsnm{Eaton},~\bfnm{M.~L.}\binits{M.~L.}}
(\byear{1987}).
\btitle{{Lectures on Topics in Probability Inequalities}}.
\bpublisher{Centrum voor Wiskunde en Informatica}.
\end{bbook}
\endbibitem

\bibitem[\protect\citeauthoryear{Eaton and Perlman}{1977}]{eaton1977reflection}
\begin{barticle}[author]
\bauthor{\bsnm{Eaton},~\bfnm{M.~L.}\binits{M.~L.}} \AND
  \bauthor{\bsnm{Perlman},~\bfnm{M.~D.}\binits{M.~D.}}
(\byear{1977}).
\btitle{{Reflection groups, generalized Schur functions, and the geometry of
  majorization}}.
\bjournal{The Annals of Probability}
\bvolume{5}
\bpages{829--860}.
\end{barticle}
\endbibitem

\bibitem[\protect\citeauthoryear{Jiang}{2009}]{Jiang2009factor}
\begin{bphdthesis}[author]
\bauthor{\bsnm{Jiang},~\bfnm{Y.}\binits{Y.}}
(\byear{2009}).
\btitle{{Factor Model Monte Carlo Methods for General Fund-of-Funds Portfolio
  Management}}
\btype{PhD thesis}, \bschool{University of Washington}.
\end{bphdthesis}
\endbibitem

\end{thebibliography}

\end{document}